\documentclass[11pt]{amsart}
\usepackage{amscd,amssymb}
\usepackage{amsthm,amsmath,amssymb}
\usepackage[matrix,arrow]{xy}

\newtheorem{theorem}[equation]{Theorem}
\newtheorem{proposition}[equation]{Proposition}
\newtheorem{lemma}[equation]{Lemma}
\newtheorem{corollary}[equation]{Corollary}

\theoremstyle{definition}
\newtheorem{definition}[equation]{Definition}

\theoremstyle{remark}
\newtheorem{remark}[equation]{Remark}

\makeatletter\@addtoreset{equation}{section} \makeatother

\begin{document}

\title{On some Fano--Enriques threefolds}

\thanks{The work was partially supported by
RFFI grant No. 08-01-00395-a and grant N.Sh.-1987.2008.1.}

\author{Ilya Karzhemanov}

\begin{abstract}
We give a classification of Fano threefolds $X$ with canonical
Gorenstein singularities such that $X$ possess a regular
involution, which acts freely on some smooth surface in $|-K_X|$,
and the linear system $|-K_X|$ gives a morphism which is not an
embedding. From this classification one gets, in particular, a
description of some natural class of Fano--Enriques threefolds.
\end{abstract}

\sloppy

\maketitle

\section{Introduction}
\label{section:introduction} In this article we use the following
\begin{definition}
\label{theorem:Fano-Enriques-def} Three-dimensional normal
projective variety $W$ with canonical singularities is called a
\emph{Fano--Enriques threefold} if the canonical divisor $K_{W}$
is not Cartier, but $-K_{W} \sim_{\mathbb{Q}} H$ for some ample
Cartier divisor $H$. The number $g: = \frac{1}{2}H^{3} + 1$ is
called \emph{genus} of $W$.
\end{definition}

In \cite{Fano} G. Fano studied three-dimensional normal projective
varieties $W$ with general hyperplane section $H$ which is a
smooth Enriques surface (see also \cite{Godeaux}). Such varieties
are always singular (see \cite{Conte-Murre}). Moreover, according
to \cite{Cheltsov-FE}, if singularities of $W$ are worse than
canonical, then $W$ is a cone over $H$. Hence, from the view point
of classification, the case when $W$ has canonical singularities
is of the main interest. If this holds, then by \cite{Cheltsov-FE}
$W$ is a Fano--Enriques threefold with isolated singularities such
that $-K_{W} \sim_{\mathbb{Q}} H$. In \cite{Fano} G. Fano was able
to obtain only partial description of such varieties (see also
\cite{Conte}, \cite{Conte-Murre}).

A new approach became possible due to the Minimal Model Program
(see \cite{Clemens-Kollar-Mori}). First of all, according to
Propositions 3.1 and 3.3 in \cite{Prokhorov-Enriques}, general
element $H_{0} \in |H|$ on a Fano--Enriques threefold $W$ has only
Du Val singularities and the minimal resolution of $H_{0}$ is a
smooth Enriques surface. From this one can deduce that
$2(K_{W}+H_{0}) \sim 0$ on $W$ (see \cite{Cheltsov-FE}). Further,
take a global log canonical cover $\pi: X \longrightarrow W$ with
respect to $K_{W} + H_{0}$ (see, for example, \cite{Kawamata}).
Here morphism $\pi$ has degree $2$ and $\pi^{*}(K_{W}+H_{0}) \sim
0$. Moreover, $\pi$ is ramified exactly at those points on $W$
where $K_{W}$ is not Cartier. Since $W$ has canonical
singularities, the number of such points is finite. In particular,
we obtain that $-K_{X} \sim \pi^{*}(H_{0})$ and $X$ is a Fano
threefold with canonical Gorenstein singularities and degree
$-K_{X}^{3} = 4g-4$. Furthermore, Galois involution of the double
cover $\pi$ induces an automorphism $\tau$ on $X$ of order $2$
such that $\tau$ acts freely in codimension $2$ and $W = X/\tau$.

The above construction has lead to the complete description of
Fano--Enriques threefolds with terminal cyclic quotient
singularities (see \cite{Bayle}, \cite{Sano}). Now let $W$ be a
Fano--Enriques threefold with isolated singularities. According to
\cite[Corollary 3.7]{Prokhorov-Enriques}, if $H^{3} \ne 2$, then
general element $H_{0} \in |H|$ is a smooth Enriques surface. In
this case on the corresponding Fano threefold $X$ there is a
$\tau$-invariant smooth $\mathrm{K3}$ surface $\pi^{*}(H_{0}) \in
|-K_{X}|$ with a free action of $\tau$.

\renewcommand{\thefootnote}{1)}

The main result of the present paper is the following
\begin{theorem}
\label{theorem:main} Let $X$ be a Fano threefold with canonical
Gorenstein singularities and $S \in |-K_{X}|$ be a smooth
$\mathrm{K3}$ surface. Suppose that there is an action of regular
involution $\tau$ on $X$ such that $\tau(S) = S$ and $\tau$ does
not have fixed points on $S$. Then
\begin{itemize}
\item the factor $X/\tau$ is a Fano--Enriques threefold with isolated singularities;
\item the linear system $|-K_{X}|$ does not have base points;
\item if the morphism $\varphi_{\scriptscriptstyle|-K_{X}|}$\footnote{for a linear system $\mathcal{L}$ we denote by
$\varphi_{\scriptscriptstyle\mathcal{L}}$ corresponding rational
map.} is not an embedding, then one has the following
possibilities:
\begin{enumerate}
\item \label{A-case} $X$ is the intersection of a quartic and a quadric in $\mathbb{P}(1,1,1,1,1,2)$, $-K_{X}^{3} = 4$;
\item \label{B-case} $X$ is the image of threefold $V$, which is a double cover of the
scroll $\mathbb{F}(d_{1},d_{2},d_{3}): =
\mathrm{Proj}\left(\bigoplus_{i=1}^{3}\mathcal{O}_{\mathbb{P}^{1}}(d_{i})\right)$
with ramification at some divisor in the linear system $|4M -
2(2-\sum_{i=1}^{3}d_{i})L|$, where $M$ is the class of
tautological divisor on $\mathbb{F}(d_{1},d_{2},d_{3})$ and $L$ is
the class of a fibre of the natural projection
$\mathbb{F}(d_{1},d_{2},d_{3}) \longrightarrow \mathbb{P}^{1}$,
under birational morphism, given by multiple anticanonical linear
system $|-rK_{V}|$, $r \gg 0$. Furthermore, for
$\left(d_{1},d_{2},d_{3}, -K_{X}^{3}\right)$ only the following
values are possible:
\begin{eqnarray}
\nonumber
(2,1,1,8),\hspace{10pt}
(2,2,2,12),\hspace{10pt}
(2,2,0,8),\hspace{10pt}
(3,1,0,8),\hspace{10pt}
(3,3,0,12),\hspace{10pt}
(4,2,0,12),\\
\nonumber
(4,4,0,16),\hspace{10pt}
(5,3,0,16),\hspace{10pt}
(6,4,0,20),\hspace{10pt}
(7,5,0,24),\hspace{10pt}
(8,6,0,28),\hspace{10pt}
\end{eqnarray}
\end{enumerate}
\end{itemize}
and each of the cases in \ref{A-case}) and in \ref{B-case}) does
occur.
\end{theorem}

From Theorem~\ref{theorem:main} and the above arguments we obtain

\begin{corollary}
\label{theorem:reduction-to-embeddings} Let $W$ be a
three-dimensional normal projective variety with general
hyperplane section which is a smooth Enriques surface. If $W$ has
canonical singularities, then it is a factor of some Fano
threefold $X$ with canonical Gorenstein singularities by the
action of regular involution on $X$, which acts freely on some
smooth surface in $|-K_{X}|$, so that one of the following holds:
\begin{itemize}
\item $X$ is one of the threefolds from Theorem~\ref{theorem:main};
\item the linear system $|-K_{X}|$ gives an embedding.
\end{itemize}
\end{corollary}

\begin{remark}
\label{remark:q-smoothing-and-ampleness} In that case when
Fano--Enriques threefold $W$ has terminal singularities there
exists a flat deformation of $W$ to Fano--Enriques threefold with
terminal cyclic quotient singularities (see \cite{Minagawa}). For
Fano threefolds $X$ in case \ref{B-case}) of
Theorem~\ref{theorem:main}, which correspond to
$\mathbb{F}(d_{1},d_{2},0)$, the same result for $W = X/\tau$ is
not known. For such $W$ it is not known also if the linear system
$|H|$ is very ample.
\end{remark}

The author would like to thank Yu. G. Prokhorov for setting the
problem and for his attention to this paper. Also the author would
like to thank I. A. Cheltsov, V. A. Iskovskikh, K. A. Shramov and
V. S. Zhgun for helpful discussions.

\section{Preliminaries}
\label{section:preliminaries} We use standard notions and facts
from the theory of minimal models and Fano varieties (see
\cite{Iskovskikh-Prokhorov}, \cite{Clemens-Kollar-Mori},
\cite{Kollar-Mori}). All varieties are assumed to be projective
and defined over $\mathbb{C}$. In what follows $X$ is a threefold
from Theorem~\ref{theorem:main}.
\begin{lemma}
\label{theorem:factor-Fano-Enriques} Factor $X/\tau$ is a
Fano--Enriques threefold with isolated singularities.
\end{lemma}

\begin{proof}
Set $W: = X/\tau$ and $\pi: X \longrightarrow W$ to be the
factorization morphism. Since $\tau$ acts freely on $S \in
|-K_{X}|$, $H: = \pi(S)$ is a smooth Enriques surface and an ample
divisor on $W$. In particular, singularities of $W$ are isolated.
From this we get $-2K_{W} \sim 2H$ (see \cite[Remark
2.8]{Cheltsov-FE}). Thus, it remains to show that $W$ has
canonical singularities.

By the above arguments $W$ is $\mathbb{Q}$-Gorenstein. Then,
according to \cite[Proposition 6.7]{Clemens-Kollar-Mori}, $W$ has
log terminal singularities. Suppose that singularities of $W$ are
worse than canonical. Then, according to \cite{Cheltsov-FE},
contraction of the negative section $E$ on the
$\mathbb{P}^{1}$-fibration $\mathbb{P}: =
\mathrm{Proj}\left(\mathcal{O}_{H}\oplus\mathcal{O}_{H}(H|_{H})\right)$
gives a birational morphism $g: \mathbb{P} \longrightarrow W$ such
that $K_{\mathbb{P}} = g^{*}(K_{W}) - E$. This implies that the
discrepancy $a(E, W)$ equals $-1$, which is a contradiction.
\end{proof}

\begin{lemma}
\label{theorem:degree-is-divisible-by-four} The degree
$-K_{X}^{3}$ is divisible by $4$.
\end{lemma}

\begin{proof}
In the notation from the proof of
Lemma~\ref{theorem:factor-Fano-Enriques}, for the ample Cartier
divisor $H = \pi(S)$ on a Fano--Enriques threefold $W$ we have
$-K_{W} \sim_{\mathbb{Q}} H$. In particular, we get: $-K_{X}^{3} =
\pi^{*}(H)^{3} = 2H^{3}$. On the other hand, according to
\cite[Lemma 2.2]{Prokhorov-Enriques}, $H^{3}$ is divisible by $2$.
Thus, $-K_{X}^{3}$ is divisible by $4$.
\end{proof}
\renewcommand{\thefootnote}{2)}

\begin{lemma}
\label{theorem:base-point-free} The linear system $|-K_{X}|$ does
not have base points.
\end{lemma}

\begin{proof}
Suppose that $B: = \mathrm{Bs}|-K_{X}| \ne
\emptyset$.\footnote{for a linear system $\mathcal{L}$ we denote
by $\mathrm{Bs}(\mathcal{L})$ its base locus.} If $\dim B = 0$,
then, according to \cite{Shin}, $B$ is a point. We have $B =
\tau(B)$ and $B \in S$. On the other hand, $\tau$ acts freely on
$S$, a contradiction.

Suppose now that $\dim B = 1$. Then, according to \cite{Shin}, we
have $B \simeq \mathbb{P}^{1}$. Thus, since $\tau(B) = B$, there
are at least two $\tau$-fixed points on $B$. On the other hand, $B
\subset S$ and $\tau$ acts freely on $S$, a contradiction.
\end{proof}

Let us consider the anticanonical morphism
$\varphi_{\scriptscriptstyle|-K_{X}|}: X \longrightarrow Y$ and
assume that it is not an isomorphism. Then
$\varphi_{\scriptscriptstyle|-K_{X}|}$ is a double cover of the
threefold $Y: = \varphi_{\scriptscriptstyle|-K_{X}|}(X) \subset
\mathbb{P}^{n}$, where $n = -\frac{1}{2}K_{X}^{3} + 2$ (see
\cite{Iskovskikh-Prokhorov}). Let us denote by $D \subset Y$ the
ramification divisor of $\varphi_{\scriptscriptstyle|-K_{X}|}$.
\begin{lemma}
\label{theorem:degree-four} Suppose that $-K_{X}^{3} = 4$. Then
$X$ is the intersection of a quartic and a quadric in
$\mathbb{P}(1,1,1,1,1,2)$.
\end{lemma}

\begin{proof}
This follows from \cite[Remark 3.2]{CPS}.
\end{proof}

\begin{remark}
\label{remark:smooth-realization} From the proof of Theorem 1.1 in
\cite{Sano} it follows that there exists a smooth Fano threefold
$X$, which is the intersection of a quartic and a quadric in
$\mathbb{P}(1,1,1,1,1,2)$, with an action of regular involution
which acts freely on some smooth surface in $|-K_{X}|$.
\end{remark}

\begin{lemma}
\label{theorem:degree-eight} Threefold $Y$ is not the cone over
Veronese surface.
\end{lemma}

\begin{proof}
Suppose that $Y$ is the cone over Veronese surface. Then the
threefold $X$ is isomorphic to a hypersurface of degree $6$ in the
weighted projective space $\mathbb{P}: = \mathbb{P}(1,1,1,2,3)$
(see \cite[Lemma 3.3]{CPS}). Since $\mathrm{Pic}(X) \simeq
\mathrm{Pic}(\mathbb{P}) = \mathbb{Z}$ (see \cite{dolgachev}),
\cite[Proposition 1.2.1]{Iskovskikh-Prokhorov} implies that for
every $m \in \mathbb{N}$ automorphism $\tau$ naturally lifts to
involution acting on the linear system
$|\mathcal{O}_{\mathbb{P}}(m)|$. This determines the lifting of
$\tau$ to involution on $\mathbb{P}$ which we again denote by
$\tau$.

Choose homogeneous coordinates $x_{0}$, $x_{1}$, $x_{2}$, $x_{3}$,
$x_{4}$ on $\mathbb{P}$, where $\deg x_{0} = \deg x_{1} = \deg
x_{2} = 1$, $\deg x_{3} = 2$, $\deg x_{4} = 3$, such that $x_{i}$
is an eigen function of $\tau$ with an eigen value $\pm 1$. After
multiplication by $-1$ and renumbering one can set $x_{0}$ and
$x_{1}$ to be $\tau$-invariant. Then the action of $\tau$ on
$\mathbb{P}$ is
\begin{equation}
\label{involution-map} [x_{0}:x_{1}:x_{2}:x_{3}:x_{4}] \mapsto
[x_{0}:x_{1}:-x_{2}:-x_{3}:-x_{4}].
\end{equation}
Indeed, in any other expression the locus of $\tau$-fixed points
on $\mathbb{P}$ contains a surface. But $\mathrm{Pic}(\mathbb{P})
= \mathbb{Z}$ and $X$ is a Cartier divisor. Hence the locus of
$\tau$-fixed points on $X$ must contain a curve which is
impossible because $\tau$ acts freely on $S \in |-K_{X}|$.

Further, according to \eqref{involution-map}, the locus of
$\tau$-fixed points on $\mathbb{P}$ consists of the curves $C_{1}
= (x_{2} = x_{3} = x_{4} = 0)$, $C_{2} = (x_{0} = x_{1} = x_{3} =
0)$ and the point $O = [0:0:0:1:0]$. Since $\tau$ acts freely on
$S \in |-K_{X}|$, we have $C_{1}$, $C_{2} \not\subset X$. This and
\eqref{involution-map} imply, since $\tau(X) = X$ and $X \in
|\mathcal{O}_{\mathbb{P}}(6)|$, that the equation of $X$ is
\begin{eqnarray}
F_{6}(x_{0}:x_{1}) + \alpha_{1}x_{2}^{6} +
x_{2}^{4}F_{2}(x_{0}:x_{1}) + \alpha_{2}x_{2}^{3}x_{4} +
x_{2}^{3}x_{3}F_{1}(x_{0}:x_{1}) +\label{element-1}\\ +
x_{2}^{2}F_{4}(x_{0}:x_{1})+x_{2}x_{3}F_{3}(x_{0}:x_{1})+
x_{2}x_{4}G_{2}(x_{0}:x_{1}) + \nonumber
\\ + x_{3}^{2}H_{2}(x_{0}:x_{1})+ x_{3}x_{4}G_{1}(x_{0}:x_{1})+
\alpha_{3}x_{4}^{2} = 0,\nonumber
\end{eqnarray}
where $\alpha_{i} \in \mathbb{C}$, $F_{i}$, $G_{i}$ are
homogeneous polynomials in $x_{0}$, $x_{1}$ of degree $i$.

On the other hand, for the $\tau$-invariant surface $S \in
|-K_{X}|$ we have $S \cap \left(C_{1} \cup C_{2} \cup \{O\}\right)
= \emptyset$ by assumption. This and \eqref{involution-map} imply,
since $\mathrm{Pic}(X) \simeq \mathrm{Pic}(\mathbb{P}) =
\mathbb{Z}$ and $-K_{X} \sim \mathcal{O}_{X}(2)$, that the
equation of $S$ on $X$ is one of the following:
\begin{eqnarray}
\alpha x_{3} + x_{2}H_{1}(x_{0}:x_{1}) = 0 \label{K-1}
\end{eqnarray}

or
\begin{eqnarray}
\beta x_{2}^{2} + H_{2}(x_{0}:x_{1}) = 0 \label{K-2},
\end{eqnarray}
where $\alpha$, $\beta \in \mathbb{C}$, $H_{i}$ are homogeneous
polynomials in $x_{0}$, $x_{1}$ of degree $i$. But in case
\eqref{K-1} one gets $S \cap C_{1} \ne \emptyset$ and in case
\eqref{K-2} we have $S \ni O$. Thus, in both cases $S$ contains a
$\tau$-fixed point. The obtained contradiction proves
Lemma~\ref{theorem:degree-eight}.
\end{proof}

\begin{remark}
\label{remark:assumption} From
Lemmas~\ref{theorem:degree-is-divisible-by-four}--\ref{theorem:degree-four},
Remark~\ref{remark:smooth-realization} and
Lemma~\ref{theorem:degree-eight} we deduce that to prove
Theorem~\ref{theorem:main} it remains to consider the case when
$-K_{X}^{3} \geqslant 8$ and the threefold $Y$ is not the cone
over Veronese surface. In what follows we assume these conditions
to be satisfied for $X$.
\end{remark}

Since the degree of $Y \subset \mathbb{P}^{n}$ equals $n-2$, by
Remark~\ref{remark:assumption} and by Enriques Theorem (see
\cite[Theorem 3.11]{iskovskikh-anti-canonical-models}) there is a
birational morphism $\varphi_{\scriptscriptstyle|M|}:
\mathbb{F}(d_{1}, d_{2}, d_{3}) \longrightarrow Y$. Here
$\mathbb{F}(d_{1}, d_{2}, d_{3}): =
\mathrm{Proj}\left(\bigoplus_{i=1}^{3}
\mathcal{O}_{\mathbb{P}^{1}}(d_{i})\right)$ is a rational scroll,
$M$ is the class of tautological divisor on $\mathbb{F}(d_{1},
d_{2}, d_{3})$, $d_{1} \geqslant d_{2} \geqslant d_{3} \geqslant
0$. Let us also denote by $L$ the class of a fibre of the natural
projection $\mathbb{F}(d_{1},d_{2},d_{3}) \longrightarrow
\mathbb{P}^{1}$.
\begin{lemma}
\label{theorem:formula-for-degree} The equality $-K_{X}^{3} =
2(d_{1}+d_{2}+d_{3})$ takes place.
\end{lemma}

\begin{proof}
We have $-\frac{1}{2}K_{X}^{3} = \deg(Y) = M^{3}$. On the other
hand, $M^{3} = d_{1}+d_{2}+d_{3}$ by \cite[A.4]{Reid-surfaces},
which implies equality we need.
\end{proof}

\begin{lemma}
\label{theorem:simple-case-d-3-non-zero} If $d_{3} \ne 0$, then
$\varphi_{\scriptscriptstyle|M|}$ is an isomorphism and $D \in |4M
- 2(\sum_{i=1}^{3}d_{i}-2)L|$. Moreover, $(d_{1},d_{2},d_{3}) =
(2,1,1)$ or $(2,2,2)$.
\end{lemma}

\begin{proof}
The fact that $\varphi_{\scriptscriptstyle|M|}$ is an isomorphism
for $d_{3} \ne 0$ follows from \cite[Theorem 2.5]{Reid-surfaces}.
Thus, we have $-K_{X} \sim
\varphi_{\scriptscriptstyle|-K_{X}|}^{*}(M)$ and $K_{Y} \sim -3M +
(d_{1}+d_{2}+d_{3}-2)L$ (see \cite[A. 13]{Reid-surfaces}). This
together with the Hurwitz formula gives $D \in |4M -
2(\sum_{i=1}^{3}d_{i}-2)L|$. Finally, since $S \in |-K_{X}|$ is a
smooth surface, the threefold $X$ has isolated singularities.
According to Table 1 in the proof of Theorem 1.5 in \cite{CPS} and
Lemmas~\ref{theorem:degree-is-divisible-by-four},
\ref{theorem:formula-for-degree}, this is possible only for
$(d_{1}, d_{2}, d_{3}) = (2,1,1)$ and $(2,2,2)$.
\end{proof}

\begin{remark}
\label{remark:assumption-1} Let $X$ be a double cover of
$\mathbb{F}(d_1,d_2,d_3)$, where $(d_{1}, d_{2}, d_{3}) = (2,1,1)$
or $(2,2,2)$, with ramification at general divisor in
$\mathcal{D}: = |4M - 2(\sum_{i=1}^{3}d_{i}-2)L|$. It is easy to
write down the basis of the linear system $\mathcal{D}$ (see
\cite[2.4]{Reid-surfaces} or \eqref{linear-system-on-scroll-0}
below) and obtain that $\mathcal{D}$ does not have base points.
This together with the Hurwitz formula implies that $X$ is a
smooth Fano threefold of degree $8$ or $12$. Moreover, according
to \cite[Remark 1.8]{CPS}, $X$ belongs to the list from Theorem
1.1 in \cite{Sano}. Hence there is an action of regular involution
$\tau$ on $X$ such that the factor $W: = X/\tau$ is a
Fano--Enriques threefold with isolated singularities. Since the
genus of $W$ equals $3$ or $4$, it follows from \cite[Corollary
3.7]{Prokhorov-Enriques} that $\tau$ acts freely on some smooth
$\mathrm{K3}$ surface in $|-K_{X}|$ (see arguments in
Introduction). Moreover, by construction the linear system
$|-K_{X}|$ gives a morphism of degree $2$.
\end{remark}
It follows from Lemma~\ref{theorem:simple-case-d-3-non-zero} and
Remark~\ref{remark:assumption-1} that to prove
Theorem~\ref{theorem:main} it remains to consider the case when
$d_{3} = 0$. In what follows we assume this condition to be
satisfied for $X$. Set $\mathbb{F}: = \mathbb{F}(d_{1},d_{2},0)$.
\begin{lemma}
\label{theorem:case-d-3-zero} In the above notation, morphism
$\varphi_{\scriptscriptstyle|M|}: \mathbb{F} \longrightarrow Y$ is
a small birational contraction and
$\varphi_{\scriptscriptstyle|M|}^{*}(D) \in |4M - 2(d_{1} + d_{2}
- 2)L|$. The exceptional locus of
$\varphi_{\scriptscriptstyle|M|}$ is an irreducible rational curve
and the threefold $Y$ is a cone with the unique singularity at the
vertex.
\end{lemma}

\begin{proof}
We have $d_{2} \ne 0$. Indeed, if $d_{2} = 0$, then $Y$ is a cone
with a curve of singularities (see the proof of Theorem 2.5 in
\cite{Reid-surfaces}). The latter implies that the singularities
of $X$ are non-isolated, which is impossible because $S \in
|-K_{X}|$ is a smooth surface. Further, as in the proof of
Lemma~\ref{theorem:simple-case-d-3-non-zero}, we obtain that
$\varphi_{\scriptscriptstyle|M|}^{*}(D) \in |4M - 2(d_{1} + d_{2}
- 2)L|$. Finally, the fact that the exceptional locus of
$\varphi_{\scriptscriptstyle|M|}$ is an irreducible rational curve
and $Y$ is a cone with the unique singularity at the vertex
follows from the proof of Theorem 2.5 in \cite{Reid-surfaces}.
\end{proof}

\begin{lemma}
\label{theorem:double-cover} In the above notation, let $V$ be the
double cover of $\mathbb{F}$ with ramification divisor
$\varphi_{\scriptscriptstyle|M|}^{*}(D)$. Then $X$ is an image of
$V$ under birational morphism, given by the multiple anticanonical
linear system $|-rK_{V}|$, $r \gg 0$.
\end{lemma}

\begin{proof}
This follows from \cite[Remark 3.8]{CPS}.
\end{proof}

Further, one has the following exact sequence:
\begin{equation}
\label{exact-sequence-aut}
1 \rightarrow G \longrightarrow \mathrm{Aut}(X) \stackrel{f}{\longrightarrow} \mathrm{Aut}(Y) \rightarrow 1,
\end{equation}
where $G$ is the group generated by Galois involution which
corresponds to $\varphi_{\scriptscriptstyle|-K_{X}|}$. Set
$\sigma: = f(\tau)$.

\begin{lemma}
\label{theorem:induced-action} In the above notation, involution
$\sigma$ lifts to the regular involution on $\mathbb{F}$.
\end{lemma}

\begin{proof}
We have $K_{\mathbb{F}} \sim -3M + (d_{1}+d_{2}-2)L$ (see \cite[A.
13]{Reid-surfaces}). Let $C \simeq \mathbb{P}^{1}$ be the
exceptional locus of $\varphi_{\scriptscriptstyle|M|}$ (see
Lemma~\ref{theorem:case-d-3-zero}). Then, since $C = M_{1} \cdot
M_{2}$ for general $M_{1} \in |M - d_{1}L|$ and $M_{2} \in |M -
d_{2}L|$ (see the proof of Theorem 2.5 in \cite{Reid-surfaces}),
we have $K_{\mathbb{F}} \cdot C = d_{1} + d_{2} - 2$ (see
\cite[A.4]{Reid-surfaces}).

If $d_{1} + d_{2} - 2 \leqslant 0$, then by
Lemma~\ref{theorem:formula-for-degree} we have $-K_{X}^{3} =
2(d_{1} + d_{2}) \leqslant 4$. This contradicts the assumption for
$-K_{X}^{3}$ (see Remark~\ref{remark:assumption}).

Now let $d_{1} + d_{2} - 2 > 0$. Then the divisor $K_{\mathbb{F}}$
is ample over $Y$. Hence $\mathbb{F}$ is a relatively minimal
model over $Y$. But every such model, which is birational to
$\mathbb{F}$, is either isomorphic to $\mathbb{F}$ or connected
with $\mathbb{F}$ by a sequence of flops over $Y$ (see
\cite[Theorem 4.3]{Kollar-flops}). Thus, in the present case all
such relatively minimal models over $Y$ are isomorphic to
$\mathbb{F}$. In particular, this holds for the
$\sigma$-equivariant canonical model of a $\sigma$-equivariant
resolution of $Y$ (see \cite{Kollar-Mori}).
\end{proof}

Let us again denote by $\sigma$ the lifting of involution $\sigma$
on $\mathbb{F}$.

\begin{lemma}
\label{theorem:sigma-invariant-linear-system} In the above
notation, linear system $|aM + bL|$ is $\sigma$-invariant on
$\mathbb{F}$ for every $a$, $b \in \mathbb{Z}$.
\end{lemma}

\begin{proof}
It follows from \cite[Lemma 2.7]{Reid-surfaces} that every divisor
$B$ on $\mathbb{F}$ is linearly equivalent to divisor $aM + bL$
for some $a$, $b \in \mathbb{Z}$. If $B$ is numerically effective,
then we have $a \geqslant 0$, since otherwise $B$ has negative
intersection with every curve in $L$. Moreover, for such $B$ we
have $b \geqslant 0$, since $M \cdot C = 0$ and $L \cdot C = 1$ in
the notation from the proof of Lemma~\ref{theorem:induced-action}.
Thus, divisors $L$ and $M$ generate the cone of numerically
effective divisors on $\mathbb{F}$. Since $\sigma$ preserves this
cone, $L^{3} = 0$ and $M^{3} > 0$ (see \cite[A.4]{Reid-surfaces}),
we obtain that the linear systems $|L|$ and $|M|$ are
$\sigma$-invariant. This implies the result we need.
\end{proof}

\begin{remark}
\label{remark:two-invariant-fibres} Since $\sigma^{*}|L| = |L|$
and $|L|$ is a pencil, there exist at least two $\sigma$-invariant
fibres $L_{0}$, $L_{1} \in |L|$ on $\mathbb{F}$.
\end{remark}

Set $M_{S}: =
\varphi_{\scriptscriptstyle|M|}^{*}(\varphi_{\scriptscriptstyle|-K_{X}|}(S))$
for the smooth $\tau$-invariant $\mathrm{K3}$ surface $S \in
|-K_{X}|$ on $X$. This is a $\sigma$-invariant divisor in $|M|$.
Set also $D': = \varphi_{\scriptscriptstyle|M|}^{*}(D)$. This is a
$\sigma$-invariant divisor in $|4M - 2(d_{1} + d_{2} - 2)L|$ (see
Lemma~\ref{theorem:case-d-3-zero} and \eqref{exact-sequence-aut}).
\begin{lemma}
\label{theorem:intersection-not-contain-fixed-points} In the above
notation, the set $D' \cap M_{S}$ does not contain $\sigma$-fixed
points.
\end{lemma}

\begin{proof}
If $D' \cap M_{S}$ contains a $\sigma$-fixed point, then the
surface $S$ contains a $\tau$-fixed point (see
\eqref{exact-sequence-aut}), which is impossible by assumption.
\end{proof}

\section{Proof of Theorem~\ref{theorem:main}}
\label{section:conclusion}
We use notation and conventions from Section~\ref{section:preliminaries}.\\
Threefold $\mathbb{F}$ is the factor of $\left(\mathbb{C}^{2}
\setminus{\{0\}}\right) \times \left(\mathbb{C}^{3}
\setminus{\{0\}}\right)$ by an action of the group
$(\mathbb{C}^{*})^{2}$ (see \cite[2.2]{Reid-surfaces}). Let us
denote by $[x_{0}:x_{1}:x_{2}]$ the projective coordinates on a
fibre $L \simeq \mathbb{P}^{2}$ of the natural projection
$\mathbb{F} \longrightarrow \mathbb{P}^{1}$. Let also
$[t_{0}:t_{1}]$ be projective coordinates on the base
$\mathbb{P}^{1}$. The functions $t_{i}$, $x_{j}$ are restrictions
of the coordinate functions on $\mathbb{C}^{2} \setminus{\{0\}}$
and $\mathbb{C}^{3} \setminus{\{0\}}$, respectively. For every
$a$, $b \in \mathbb{Z}$ it then follows that linear system
$|aM+bL|$ is generated by polynomials of the form
\begin{equation}
\label{linear-system-on-scroll-0}
g_{i_{1},i_{2},i_{3}}x_{0}^{i_{1}}x_{1}^{i_{2}}x_{2}^{i_{3}},
\end{equation}
where $i_{1}+i_{2}+i_{3} = a$, $i_{j} \geqslant 0$,
$g_{i_{1},i_{2},i_{3}}: = g_{i_{1},i_{2},i_{3}}(t_{0}:t_{1})$ is a
homogeneous polynomial of degree $b+d_{1}i_{1}+d_{2}i_{2}
\geqslant 0$ (see \cite[2.4]{Reid-surfaces}).

\begin{lemma}
\label{theorem:irreducible} General element in the linear system
$|4M - 2(d_{1} + d_{2} - 2)L|$ is irreducible.
\end{lemma}

\begin{proof}
Let general element in $\mathcal{R}:=|4M - 2(d_{1} + d_{2} - 2)L|$
be reducible. Then, according to Table 1 in the proof of Theorem
1.5 in \cite{CPS}, we have $d_{1}
> d_{2}$, and $\mathcal{R}$ is generated by polynomials in \eqref{linear-system-on-scroll-0} with $a = 4$, $b = 2(2 - d_{1} -
d_{2})$ and $i_{1} > 0$. In particular, divisor $D' =
\varphi_{\scriptscriptstyle|M|}^{*}(D)$ contains the surface $R
\in |M - d_{1}L|$, given by equation $x_{0} = 0$.

Since $d_{1} > d_{2}$, it follows from
\eqref{linear-system-on-scroll-0} that the linear system $|M -
d_{1}L|$ is generated by $x_{0}$. Then by
Lemma~\ref{theorem:sigma-invariant-linear-system} we obtain that
$R = \sigma(R)$. Let $L_{0}$, $L_{1} \in |L|$ be two
$\sigma$-invariant fibres on $\mathbb{F}$ (see
Remark~\ref{remark:two-invariant-fibres}). Then $R|_{L_{i}}$ and
$M_{S}|_{L_{i}}$ are $\sigma$-invariant lines on $L_{i} \simeq
\mathbb{P}^{2}$, $i \in \{0,1\}$. In particular, the sets $R \cap
M_{S} \cap L_{i}$ contain at least one $\sigma$-fixed point each.
But $R \cap M_{S} \cap L_{i} \subset D' \cap M_{S}$. Thus, we get
a contradiction with
Lemma~\ref{theorem:intersection-not-contain-fixed-points}.
\end{proof}

According to Table 1 in the proof of Theorem 1.5 in \cite{CPS} and
Lemmas~\ref{theorem:degree-is-divisible-by-four},
\ref{theorem:formula-for-degree}, \ref{theorem:irreducible} one
gets only the following possibilities for $(d_{1},d_{2})$:
\begin{equation}
\label{possibilities-for-d-i}
(2,2),\hspace{5pt} (3,1),\hspace{5pt} (3,3),\hspace{5pt} (4,2),\hspace{5pt} (4,4),\hspace{5pt} (5,3),\hspace{5pt} (6,4),
\hspace{5pt} (7, 5),\hspace{5pt} (8,6).
\end{equation}
This and Lemmas~\ref{theorem:factor-Fano-Enriques},
\ref{theorem:double-cover} imply that to prove
Theorem~\ref{theorem:main} it remains to show that for every pair
$(d_{1},d_{2})$ in \eqref{possibilities-for-d-i} there is a Fano
threefold $X$ with canonical Gorenstein singularities such that
$X$ possess a regular involution, which acts freely on some smooth
$\mathrm{K3}$ surface in $|-K_{X}|$, and the linear system
$|-K_{X}|$ gives a morphism which is not an embedding.

Set $\mathbb{F}: = \mathbb{F}(d_{1},d_{2},0)$ for $(d_{1},d_{2})$
in \eqref{possibilities-for-d-i}. Let us use previous notation for
coordinates on the base $\mathbb{P}^{1}$ and on a fibre $L \simeq
\mathbb{P}^{2}$ of the natural projection $\mathbb{F}
\longrightarrow \mathbb{P}^{1}$. We define regular involution
$\sigma$ on $\mathbb{F}$ by the following relations:
\begin{equation}
\label{precise-involution-1}
\sigma^{*}(t_{0}) = t_{0},\qquad \sigma^{*}(t_{1}) = -t_{1}
\end{equation}

and
\begin{equation}
\label{precise-involution-2}
\sigma^{*}(x_{0}) = -x_{0},\qquad \sigma^{*}(x_{1}) = x_{1},\qquad \sigma^{*}(x_{2}) = -x_{2}.
\end{equation}

\begin{remark}
\label{remark:fully-determined} Since $t_{i}$, $x_{j}$ are
restrictions of the coordinate functions on $\mathbb{C}^{2}
\setminus{\{0\}}$ and $\mathbb{C}^{3} \setminus{\{0\}}$,
respectively, \eqref{precise-involution-1} and
\eqref{precise-involution-2} commute with the action of the group
$(\mathbb{C}^{*})^{2}$, the action of $\sigma$ on $\mathbb{F}$ is
completely determined by relations \eqref{precise-involution-1}
and \eqref{precise-involution-2}. On the other hand, from
Lemma~\ref{theorem:sigma-invariant-linear-system} it is easy to
see that up to the sign change every regular involution on
$\mathbb{F}$ is determined by relations of the form
\eqref{precise-involution-1} and \eqref{precise-involution-2}.
\end{remark}

Let us denote by $C$ the curve on $\mathbb{F}$, given by equations
$x_{0} = x_{1} = 0$. We prove the following
\begin{proposition}
\label{theorem:linear-systems-D-M} In the above notation, there
are linear systems $\mathcal{D} \subseteq |4M - 2(d_{1} + d_{2} -
2)L|$, $\mathcal{M} \subseteq |M|$, where $M$ is the class of
tautological divisor on $\mathbb{F}$, such that
\begin{itemize}
\item $\dim \mathcal{D}$, $\dim \mathcal{M} > 0$;
\item $\mathcal{D}$ consists of $\sigma$-invariant divisors, $\mathrm{Bs}(\mathcal{D}) = C$ and $\mathrm{mult}_{C}(\mathcal{D}) \leqslant 3$;
\item $\mathcal{M}$ consists of $\sigma$-invariant divisors and $\mathrm{Bs}(\mathcal{M}) \cap C = \emptyset$;
\item double cover of $\mathbb{F}$ with ramification at general divisor in $\mathcal{D}$ has canonical
singularities;
\item for general divisors $D_{0} \in \mathcal{D}$, $M_{0} \in \mathcal{M}$ and the set of $\sigma$-fixed points $\mathbb{F}^{\sigma}$
on $\mathbb{F}$ we have $M_{0} \cap D_{0} \cap \mathbb{F}^{\sigma}
= \emptyset$.
\end{itemize}
\end{proposition}

\begin{proof}
The conditions $\sigma(D_{0}) = D_{0}$, $\mathrm{Bs}(\mathcal{D})
= C$, $\mathrm{mult}_{C}(\mathcal{D}) \leqslant 3$ and
\eqref{linear-system-on-scroll-0}, \eqref{precise-involution-1},
\eqref{precise-involution-2} imply that the equation of general
divisor $D_0\in\mathcal{D}$ for $(d_{1},d_{2})$ in
\eqref{possibilities-for-d-i} must be one of the following:

\begin{center}
\begin{tabular}{|c|c|}
\hline
$(d_1, d_2)$ & \mbox{equation of $D_0$}\\
\hline
$(2,2)$ & $\alpha x_{0}^{2}x_{2}^{2} + \beta x_{1}^{2}x_{2}^{2} + \gamma t_{0}^{4}x_{0}^{4} + \gamma't_{1}^{4}x_{0}^{4} +
\delta t_{0}^{4}x_{1}^{4} + \delta't_{1}^{4}x_{1}^{4} + P_{1} = 0$\\
\hline
$(3,1)$ & $\alpha t_{0}^{2}x_{0}^{2}x_{2}^{2} + \alpha' t_{1}^{2}x_{0}^{2}x_{2}^{2} + \beta x_{1}^{4} + \gamma t_{0}^{8}x_{0}^{4} +
\gamma't_{1}^{8}x_{0}^{4} + P_{2} = 0$\\
\hline
$(3,3)$ & $\alpha t_{0}x_{0}^{3}x_{2} + \beta t_{1}x_{1}^{3}x_{2} + \gamma t_{0}^{4}x_{0}^{4} + \gamma' t_{1}^{4}x_{0}^{4}
+ \delta t_{0}^{4}x_{1}^{4} + \delta' t_{1}^{4}x_{1}^{4} + P_{3} = 0$\\
\hline
$(4,2)$ & $\alpha x_{0}^{2}x_{2}^{2} + \beta x_{1}^{4} + \gamma t_{0}^{8}x_{0}^{4} + \gamma't_{1}^{8}x_{0}^{4} + P_{4} = 0$\\
\hline
$(4,4)$ & $\alpha x_{0}^{3}x_{2} + \beta t_{0}^{4}x_{0}^{4} + \beta' t_{1}^{4}x_{0}^{4} +
\gamma t_{0}^{4}x_{1}^{4} + \gamma't_{1}^{4}x_{1}^{4} + P_{5} = 0$\\
\hline
$(5,3)$ & $\alpha t_{0}^{3}x_{0}^{3}x_{2} + \beta t_{1}x_{0}^{2}x_{1}x_{2} +
\gamma x_{1}^{4} + \delta t_{0}^{8}x_{0}^{4} + \delta't_{1}^{8}x_{0}^{4} + P_{6} = 0$\\
\hline
$(6,4)$ & $\alpha t_{0}^{2}x_{0}^{3}x_{2} + \alpha' t_{1}^{2}x_{0}^{3}x_{2} + \beta x_{1}^{4} + \gamma t_{0}^{8}x_{0}^{4} + \gamma't_{1}^{8}x_{0}^{4} + P_{7} = 0$\\
\hline
$(7,5)$ & $\alpha t_{0}x_{0}^{3}x_{2} + \beta x_{1}^{4} +
\gamma t_{0}^{8}x_{0}^{4} + \gamma't_{1}^{8}x_{0}^{4} + P_{8} = 0$\\
\hline
$(8,6)$ & $\alpha x_{0}^{3}x_{2} + \beta x_{1}^{4} + \gamma t_{0}^{8}x_{0}^{4} + \gamma't_{1}^{8}x_{0}^{4} + P_{9} = 0$\\
\hline
\end{tabular}
\begin{center}
\mbox{{\small Table 1.}}
\end{center}
\end{center}

Throughout the Table 1 $\alpha$, $\beta$, $\gamma$, $\delta$,
$\alpha'$, $\beta'$, $\gamma'$, $\delta' \in \mathbb{C}$, $P_{i}:
= P_{i}(t_{0},t_{1},x_{0},x_{1},x_{2})$ is a polynomial of degree
$\geqslant 3$ in $x_{0}$, $x_{1}$ such that $\sigma^{*}(P_{i}) =
P_{i}$ and $P_{i}(t_{0},t_{1},0,0,x_{2}) = 0$ for $1 \leqslant i
\leqslant 9$.

\begin{lemma}
\label{theorem:canonical-singularities-for-double-covers} Double
cover of $\mathbb{F}$ with ramification at general divisor in
$\mathcal{D}$ has canonical singularities.
\end{lemma}

\begin{proof}
According to \cite[Corollary 2.7]{CPS} and condition
$\mathrm{Bs}(\mathcal{D}) = C$, it is enough to show that for
every point $p$ on the curve $C$ there is a divisor $D_{0} \in
\mathcal{D}$ such that the double cover $\varphi: V
\longrightarrow \mathbb{F}$ of $\mathbb{F}$ with ramification at
$D_{0}$ has canonical singularity at the point $o: =
\varphi^{-1}(p)$.

Put $x_{0} = y$, $x_{1} = z$, $x_{2} = 1$ in equations from Table
1. We obtain:

\begin{center}
\begin{tabular}{|c|c|}
\hline
$(d_1, d_2)$ & \mbox{equation of $V$ in the neighborhood of $o$ with local coordinates $x$, $y$, $z$}\\
\hline
$(2,2)$ & $x^{2} + \alpha y^{2} + \beta z^{2} + \gamma t_{0}^{4}y^{4} + \gamma't_{1}^{4}y^{4} +
\delta t_{0}^{4}z^{4} + \delta't_{1}^{4}z^{4} + Q_{1} = 0$\\
\hline
$(3,1)$ & $x^{2} + \alpha t_{0}^{2}y^{2} + \alpha' t_{1}^{2}y^{2} + \beta z^{4} + \gamma t_{0}^{8}y^{4} + \gamma't_{1}^{8}y^{4} + Q_{2} = 0$\\
\hline
$(3,3)$ & $x^{2} + \alpha t_{0}y^{3} + \beta t_{1}z^{3} + \gamma t_{0}^{4}y^{4} + \gamma' t_{1}^{4}y^{4}
+ \delta t_{0}^{4}z^{4} + \delta' t_{1}^{4}z^{4}+ Q_{3} = 0$\\
\hline
$(4,2)$ & $x^{2} + \alpha y^{2} + \beta z^{4} + \gamma t_{0}^{8}y^{4} + \gamma't_{1}^{8}y^{4} + Q_{4} = 0$\\
\hline
$(4,4)$ & $x^{2} + \alpha y^{3} + \beta t_{0}^{4}y^{4} + \beta' t_{1}^{4}y^{4} + \gamma t_{0}^{4}z^{4} + \gamma' t_{1}^{4}z^{4}
+ Q_{5} = 0$\\
\hline
$(5,3)$ & $x^{2} + \alpha t_{0}^{3}y^{3} + \beta t_{1}y^{2}z + \gamma z^{4} + \delta t_{0}^{8}y^{4} + \delta't_{1}^{8}y^{4} + Q_{6} = 0$\\
\hline
$(6,4)$ & $x^{2} + \alpha t_{0}^{2}y^{3} + \alpha' t_{1}^{2}y^{3} +
\beta z^{4} + \gamma t_{0}^{8}y^{4} + \gamma' t_{1}^{8}y^{4} + Q_{7} = 0$\\
\hline
$(8,6)$ & $x^{2} + \alpha y^{3} + \beta z^{4} + \gamma t_{0}^{8}y^{4} + \gamma't_{1}^{8}y^{4} + Q_{9} = 0$\\
\hline
\end{tabular}

\begin{center}
\mbox{{\small Table 2.}}
\end{center}
\end{center}

Throughout the Table 2 $Q_{i}: = P_{i}(t_{0},t_{1},x,y,1)$. It
follows that for $(d_{1},d_{2}) \ne (7,5)$ for every point $p =
[t_0:t_1]$ on the curve $C$ there is a divisor $D_{0} \in
\mathcal{D}$ such that $o = \varphi^{-1}(p) \in V$ is a
$\mathrm{cDV}$ singularity and hence canonical (see
\cite{Reid-canonical-threefolds}).

For $(d_{1},d_{2}) = (7,5)$ in the neighborhood of $o$ with local
coordinates $x$, $y$, $z$ threefold $V$ is given by equation (see
Table 1):
\begin{eqnarray}
\label{equation-in-7-5-case}
x^{2} + \alpha t_{0}y^{3} + \beta z^{4}
+ \gamma t_{0}^{8}y^{4} + \gamma't_{1}^{8}y^{4} + Q_{8} = 0,
\end{eqnarray}
where $Q_{8}: = P_{8}(t_{0},t_{1},x,y,1)$. If $p = [t_0:t_1]$ is a
point on the curve $C$ with $t_{0} \ne 0$, then one may put $t_0 =
1$, $t_1 = t$ and find the equation of $V$ in the neighborhood of
$o$ with local coordinates $x$, $y$, $z$, $t$:
$$
x^{2} + \alpha y^{3} + \beta z^{4} + \gamma y^{4} +
\gamma't^{8}y^{4} + Q'_{8} = 0,
$$
where $Q'_{8}: = Q_{8}(1,t,x,y,1)$. Then \cite[Theorem 2.10]{CPS}
implies that for general divisor $D_{0}$ singularity $o \in V$ is
$\mathrm{cE_{6}}$.

Now let $p = [0:1]$. Then in \eqref{equation-in-7-5-case} one may
put $t_{0} = t$, $t_{1} = 1$ and find the equation of $V$ in the
neighborhood of $o$ with local coordinates $x$, $y$, $z$, $t$:
\begin{equation}
\nonumber
x^{2} + \alpha ty^{3} + \beta z^{4} + \gamma t^{8}y^{4} + \gamma' y^{4} + Q'_{8} = 0,
\end{equation}
where $Q'_{8}: = Q_{8}(t,1,x,y,1)$. It is easy to see that the
weighted blow up $\widetilde{V} \longrightarrow V$ at the point
$o$ with weights (2,~1,~1,~1) is crepant (see the proof of Theorem
2.11 in \cite{Reid-canonical-threefolds}) and the threefold
$\widetilde{V}$ is smooth. Hence for general divisor $D_{0}$
singularity $o \in V$ is canonical.
Lemma~\ref{theorem:canonical-singularities-for-double-covers} is
completely proved.
\end{proof}

Further, the conditions $\sigma(M_{0}) = M_{0}$,
$\mathrm{Bs}(\mathcal{M}) \cap C = \emptyset$ and
\eqref{linear-system-on-scroll-0}, \eqref{precise-involution-1},
\eqref{precise-involution-2} imply that the equation of general
divisor $M_{0} \in \mathcal{M}$ for $(d_{1},d_{2})$ in
\eqref{possibilities-for-d-i} must be one of the following:
\begin{center}
\begin{tabular}{|c|c|}
\hline
$(d_1, d_2)$ & \mbox{equation of $M_0$}\\
\hline
$(2,2)$ & $at_{0}^{2}x_{0} + bt_{1}^{2}x_{0} + cx_{2} + F_{1} = 0$\\
\hline
$(3,1)$ & $at_{0}^{3}x_{0} + bt_{1}x_{1} + cx_{2} + F_{2} = 0$\\
\hline
$(3,3)$ & $at_{0}^{3}x_{0} + bt_{1}^{3}x_{1} + cx_{2} + F_{3} = 0$\\
\hline
$(4,2)$ & $at_{0}^{4}x_{0} + bt_{1}^{4}x_{0} + cx_{2} + F_{4} = 0$\\
\hline
$(4,4)$ & $at_{0}^{4}x_{0} + bt_{1}^{4}x_{0} + cx_{2} + F_{5} = 0$\\
\hline
$(5,3)$ & $at_{0}^{5}x_{0} + bt_{1}^{3}x_{1} + cx_{2} + F_{6} = 0$\\
\hline
$(6,4)$ & $at_{0}^{6}x_{0} + bt_{1}^{6}x_{0} + cx_{2} + F_{7} = 0$\\
\hline
$(7,5)$ & $at_{0}^{7}x_{0} + bt_{1}^{5}x_{1} + cx_{2} + F_{8} = 0$\\
\hline
$(8,6)$ & $at_{0}^{8}x_{0} + bt_{1}^{8}x_{0} + cx_{2} + F_{9} = 0$\\
\hline
\end{tabular}

\begin{center}
\mbox{{\small Table 3.}}
\end{center}
\end{center}

Throughout the Table 3 $a$, $b$, $c \in \mathbb{C}$,
$F_{i}:=F_{i}(t_{0},t_{1},x_{0},x_{1})$ is a polynomial of degree
$1$ in $x_{0}$, $x_{1}$ such that $\sigma^{*}(F_{i}) = -F_{i}$ and
$F_{i}(t_{0},t_{1},0,0) = 0$ for $1 \leqslant i \leqslant 9$.

\begin{lemma}
\label{theorem:empty-intersection} For general divisors $D_{0} \in
\mathcal{D}$, $M_{0} \in \mathcal{M}$ and the set of
$\sigma$-fixed points $\mathbb{F}^{\sigma}$ on $\mathbb{F}$ we
have $M_{0} \cap D_{0} \cap \mathbb{F}^{\sigma} = \emptyset$.
\end{lemma}

\begin{proof}
From \eqref{precise-involution-1}, \eqref{precise-involution-2} we
obtain the equations for $\mathbb{F}^{\sigma}$:
\begin{equation}
\nonumber
t_{0}t_{1} = x_{1}x_{0} = x_{1}x_{2} = 0.
\end{equation}
This implies that $\mathbb{F}^{\sigma} = l_{1} \cup l_{2} \cup
O_{1} \cup O_{2}$, where $l_{i} = (t_{i} = x_{1} = 0)$ and $l_{i}
\not \ni O_{i} = (t_{i} = x_{0} = x_{2})$ are a curve and a point
on the fibre $L_{i} = (t_{i} = 0)$, $i \in \{0, 1\}$,
respectively.

It follows from equations in Tables 1 and 3 that $O_{i} =
\mathrm{Bs}\left(\mathcal{M}|_{L_{i}}\right)$, $O_{i} \not\in
D_{0}$ and the set $D_{0} \cap l_{i}$ is finite, $i \in \{0, 1\}$.
Then, since $\mathcal{M}|_{L_{i}}$ is a pencil of lines on $L_{i}
\simeq \mathbb{P}^{2}$, we obtain that $M_{0} \cap D_{0} \cap
\mathbb{F}^{\sigma} = \emptyset$.
\end{proof}

From
Lemmas~\ref{theorem:canonical-singularities-for-double-covers},
\ref{theorem:empty-intersection} and Tables 1, 3 we obtain the
assertion of Proposition~\ref{theorem:linear-systems-D-M}.
\end{proof}

Let $\mathcal{D}$, $\mathcal{M}$ be the linear systems from
Proposition~\ref{theorem:linear-systems-D-M} and $D_{0} \in
\mathcal{D}$, $M_{0} \in \mathcal{M}$ be general divisors. Let us
denote by $\varphi: V \longrightarrow \mathbb{F}$ the double cover
of $\mathbb{F}$ with ramification at $D_{0}$. By
Proposition~\ref{theorem:linear-systems-D-M} threefold $V$ has
canonical singularities. Moreover, from the Hurwitz formula we
obtain
\begin{equation}
\label{hurwitz}
-K_{V} \sim \varphi^{*}(M).
\end{equation}
Thus, $V$ is a weak Fano threefold with canonical Gorenstein
singularities. Furthermore, by construction, $V$ possess a regular
involution $\theta$, which acts non trivially on the fibres of
$\varphi$, such that the restriction of $\theta$ on $\mathbb{F}$
coincides with $\sigma$.

Further, \cite[Theorem 3.3]{Kollar-Mori}, \eqref{hurwitz} and
Lemma~\ref{theorem:case-d-3-zero} imply that the linear system
$|-rK_{V}|$, $r \gg 0$, gives a birational morphism $\psi: V
\longrightarrow X$ such that $\psi$-exceptional locus is the curve
$\varphi^{-1}(C)$ and $X$ is a Fano threefold with canonical
Gorenstein singularities which possesses a regular involution
$\tau$, the restriction of $\theta$.

It follows from \eqref{hurwitz} and
Proposition~\ref{theorem:linear-systems-D-M} that $S: =
\varphi^{*}(M_{0}) \in |-K_{V}|$ is a smooth $\mathrm{K3}$ surface
with a free action of involution $\theta$ such that $S \cap
\varphi^{-1}(C) = \emptyset$. This implies that $\psi(S) \in
|-K_{X}|$ is a smooth $\mathrm{K3}$ surface with a free action of
involution $\tau$. Finally, according to \cite{CPS}, the linear
system $|-K_{X}|$ gives a morphism which is not an embedding. This
completes the construction of Fano threefolds, which satisfy the
conditions of Theorem~\ref{theorem:main}, for $(d_{1},d_{2})$ in
\eqref{possibilities-for-d-i}.

Theorem~\ref{theorem:main} is completely proved.

\end{document}